\documentclass{amsart}
\usepackage{amsmath,amssymb,mathrsfs,hyperref,color}
\usepackage{subfigure}
\usepackage{float} 
\usepackage{times}
\usepackage{tikz}
\usetikzlibrary{intersections}
\usepackage{paralist}

\makeatletter
\def\setliststart#1{\setcounter{\@listctr}{#1}%
  \addtocounter{\@listctr}{-1}}
\makeatother

\makeatletter
\@addtoreset{figure}{section}
\makeatother

\usepackage{calc}
 \newtheorem{The}{Theorem}[section]
 \newtheorem{Cor}[The]{Corollary}
 \newtheorem{Lem}[The]{Lemma}
 \newtheorem{Pro}[The]{Proposition}
 \theoremstyle{definition}
 
 \theoremstyle{remark}
 \newtheorem{Rem}[The]{Remark}
 
 \numberwithin{equation}{section}

\newcommand{\R}{\mathbb{R}}

\title[On the vanishing contact structure for viscosity solutions]{On the vanishing contact structure for viscosity solutions of contact type Hamilton-Jacobi equations I:\\ Cauchy problem}
\author{Kai Zhao \and Wei Cheng}
\address{Department of Mathematics, Nanjing University, Nanjing 210093, China}
\email{15251879427@163.com}
\address{Department of Mathematics, Nanjing University, Nanjing 210093, China}
\email{chengwei@nju.edu.cn}
\date{\today}
\begin{document}
\maketitle

\begin{abstract}
	We study the representation formulae for the fundamental solutions and viscosity solutions of the Hamilton-Jacobi equations of contact type. We also obtain a vanishing contact structure result for relevant Cauchy problems which can be regarded as an extension to the vanishing discount problem. 
\end{abstract}

\section{Introduction}

In the previous work (\cite{Wang-Wang-Yan1} and \cite{Wang-Wang-Yan2}), the authors developed an analogy to weak KAM theory for contact systems on compact manifolds. This leads to a representation formula of the viscosity solutions of the Hamilton-Jacobi equation
\begin{equation}\label{eq:intro_HJe}\tag{HJ$_e$}
\begin{cases}
	D_tu(t,x)+H(x,u(t,x),D_xu(t,x))=0,& (t,x)\in(0,+\infty)\times M\\
	u(0,x)=\phi(x),& x\in M,
\end{cases}
\end{equation}
and
\begin{equation}\label{eq:intro_HJs}\tag{HJ$_s$}
H(x,u(x),Du(x))=c,\quad x\in M,
\end{equation}
using \textit{implicit variational principle}, where $M$ is a $C^2$ connected closed manifold and $c$ is contained in the set of critical values. The power of this celebrating work has been shown to understand certain systems in a much wider and deeper viewpoint. The main purpose of this paper is to understand the limit of the viscosity solutions of \eqref{eq:intro_HJe} in the case $M=\R^n$ when $H_u$ is uniformly bounded and tends to $0$.

For the special case when $H$ having the form of a s Hamiltonian $H_0$ with a discount factor $\lambda>0$, i.e., $H=\lambda u+H_0(x,p)$, this problem has been widely studies. From calculus of variations and optimal controls point of view, the associated Lagrangian $L=-\lambda u+L_0(x,v)$, where $\lambda>0$ and $L_0$ is a Tonelli Lagrangian. A classical problem in ergodic control consists of studying the limit behavior of the optimal value $u_{\lambda}$ of a discounted cost functional with infinite horizon as the discount factor $\lambda$ tends to zero. In the literature, this problem has been addressed under various conditions ensuring that the rescaled value function $\lambda u_{\lambda}$ converges uniformly to a constant limit. 

In recent works, for instance, \cite{DFIZ}, \cite{IMT1} and \cite{IMT2}, the behavior of the vanishing discount limit has been widely studied in the compact manifold case, especially applying Aubry-Mather theory and weak KAM theory. Sufficient evidences show that the method we use in this paper is also a way to understand the vanishing contact structure limit by developing the Aubry-Mather theory for contact type systems \eqref{eq:intro_HJs} under suitable conditions.

We will use a Langrangian approach of the solutions of \eqref{eq:intro_HJe} in the viscosity sense developed in \cite{CCY} using the {\em generalized variational principle} proposed by Gustav Herglotz in 1930 (see \cite{CCY} and the references therein). 

Let $H:\R^n\times\R\times\R^n\to\R$ be a function of class $C^2$ satisfying the following conditions:
\begin{enumerate}[(H1)]
  \item $H_{pp}(x,r,p)>0$ for all $(x,r,p)\in \R^n\times\R\times\R^n$;
  \item For each $r\in\R$, there exist two superlinear and nondecreasing function $\overline{\Theta}_r,\Theta_r:[0,+\infty)\to[0,+\infty)$ and $C_r>0$, such that
  $$
  \overline{\Theta}_r(|p|)\geqslant H(x,r,p)\geqslant\Theta_r(|p|)-C_r,\quad (x,p)\in\R^n\times\R^n.
  $$
  \item There exists $K>0$ such that
  $$
  |H_r(x,r,p)|\leqslant K,\quad (x,r,p)\in \R^n\times\R\times\R^n.
  $$
\end{enumerate}
It is natural to introduce the associated Lagrangian
\begin{equation}\label{eq:intro_L}
	L(x,r,v)=\sup_{p\in\R^n}\{\langle p,v\rangle-H(x,r,p)\},\quad (x,r,v)\in \R^n\times\R\times\R^n.
\end{equation}
Set $L_0(x,v)=L(x,0,v)$. 

Fix $x,y\in\R^n$, $u\in\R$ and $t>0$. Let $\xi\in\Gamma^t_{x,y}$, we consider the Carath\'eodory equation 
\begin{equation}\label{eq:caratheodory_L}
	\dot{u}_{\xi}(s)=L(\xi(s),u_{\xi}(s),\dot{\xi}(s)),\quad a.e.\ s\in[0,t]
\end{equation}
with initial conditions $u_{\xi}(0)=u$. We define
\begin{equation}\label{eq:fundamental_solution}
	h(t,x,y,u):=\inf_{\xi}\int^t_0L(\xi(s).u_{\xi}(s),\dot{\xi}(s))\ ds,\quad A(t,x,y,u):=h(t,x,y,u)+u,
\end{equation}
where the infimum is taken over $\xi\in\Gamma^t_{x,y}=\{\xi:W^{1,1}([0,t],\R^n): \xi(0)=x, \xi(t)=y\}$ and $u_{\xi}:[0,t]\to\R^n$ is a absolutely continuous curve determined by \eqref{eq:caratheodory_L}. Problem \eqref{eq:fundamental_solution} under the constraint \eqref{eq:caratheodory_L} is called \textit{generalized variational principle} of Herglotz (\cite{GGG}, \cite{GG1} and \cite{GG2}). We call $h(t,x,y,u)$ defined in \eqref{eq:fundamental_solution} the {\em fundamental solution}. It is clear that \eqref{eq:caratheodory_L} admits a unique absolutely continuous solution $u_{\xi}(s)=u_{\xi}(s;u)$ on $[0,t]$, by Proposition \ref{caratheodory}, if there exists $f\in L^1([0,t],\R)$ such that $|L(\xi(s),u,\dot{\xi}(s))|\leqslant f(s)$, a.e., $s\in[0,t]$ . In \cite{CCY}, the authors showed that 

\begin{Pro}\label{Herglotz_Lie}
Let $L$ satisfy conditions \mbox{\rm (L1)-(L3)}, then for fixed $x,y\in\R$, $t>0$ and $u\in\R$, the problem \eqref{eq:fundamental_solution} admits a minimizer. Moreover,
\begin{enumerate}[\rm (a)]
  \item Both $\xi$ and $u_{\xi}$ are of class $C^2$ and $\xi$ satisfies Herglotz equation
  \begin{equation}\label{eq:Herglotz}
	\begin{split}
		&\,\frac d{ds}L_v(\xi(s),u_{\xi}(s),\dot{\xi}(s))\\
	=&\,L_x(\xi(s),u_{\xi}(s),\dot{\xi}(s))+L_u(\xi(s),u_{\xi}(s),\dot{\xi}(s))L_v(\xi(s),u_{\xi}(s),\dot{\xi}(s)).
	\end{split}
  \end{equation}
       for all $s\in[0,t]$ where $u_{\xi}$ is the unique solution of \eqref{eq:caratheodory_L};
  \item Let $p(s)=L_v(\xi(s),u_{\xi}(s),\dot{\xi}(s))$ be the dual arc, then $p$ is also of class $C^2$ and we conclude that $(\xi,p,u_{\xi})$ satisfies Lie equation
  \begin{equation}
  	\begin{cases}
  		\dot{\xi}(s)=H_p(\xi(s),u_{\xi}(s),p(s));\\
  		\dot{p}(s)=-H_x(\xi(s),u_{\xi}(s),p(s))-H_u(\xi(s),u_{\xi}(s),p(s))p(s);\\
  		\dot{u}_{\xi}(s)=p(s)\cdot\dot{\xi}(s)-H(\xi(s),u_{\xi}(s),p(s)).
  	\end{cases}
  \end{equation}
\end{enumerate}
\end{Pro}

Such a Lagrangian approach also leads to a very clear explanation of the representation formulae of the value function of the associated problem from calculus of variations. That is, due to Proposition \ref{Herglotz_Lie}, we conclude that the relevant fundamental solution 
\begin{align*}
	h(t,x,y,u)=\inf_{\xi\in\Gamma^t_{x,y}\cap C^{2}([0,t])}u_{\xi}(t)-u,\quad t>0,\ x,y\in\R^n, u\in\R,
\end{align*}
and $u_\xi$ is uniquely determined by \eqref{eq:caratheodory_L} in classical sense. By solving the ordinary differential equation \eqref{eq:caratheodory_L}, we can have some new representation formulae of the viscosity solution $u(t,x)$ of \eqref{eq:intro_HJe}. \begin{equation}\label{eq:fundamental_solution2}
	u(t,x)=\inf_{\xi\in\Gamma^t_{y,x}}\left\{e^{\int^t_0L_u d\tau}\phi(\xi(0))+\int^t_0e^{\int^t_sL_u d\tau}(L-u_{\xi}\cdot L_u)\ ds\right\},
	\end{equation}
	where $u_{\xi}$ is uniquely determined by \eqref{eq:caratheodory_L} with $u_{\xi}(t)=u$.

This approach also leads to a result on the vanishing contact structure limit problem. This can be regarded as a generalization of the vanishing discount problem in PDE and control theory.

\medskip

\noindent{\bf Main Result I:} Suppose that $\{L^{\lambda}\}_{\lambda>0}$ is a family of Tonelli Lagrangians satisfying conditions \mbox{\rm (L1)}, \mbox{\rm (L2)} and \mbox{\rm (L3')} at the beginning of section 2, with $\{H^{\lambda}\}$ the family of associated Tonelli Hamiltonians. Let each $u^{\lambda}$ be the unique viscosity solution of \eqref{eq:intro_HJe} with respect to $H^{\lambda}$ and $u$ defined by \eqref{eq:value_function_no_u} be the unique viscosity solution of \eqref{eq:HJe1}. If $\phi$ is Lipschitz and bounded, then
\begin{align*}
	\lim_{\lambda\to0^+}u^{\lambda}(t,x)=u(t,x),\quad (t,x)\in(0,+\infty)\times\R^n.
\end{align*}

\noindent{\bf Main Result II:}  Under the same assumptions as above and replacing \mbox{\rm (L3')} by \mbox{\rm (L3'')} (at the beginning of section 2), then $u^{\lambda}$ tends to $u$ uniformly as $\lambda\to0^+$ on any compact subset of $(0,+\infty)\times\R^n$.	

\medskip

This paper is organized as follows. In Section 2.1, we give a representation formula for the equation \eqref{eq:intro_HJe}. In section 2.2, we discuss our vanishing contact structure results for \eqref{eq:intro_HJe}.

\medskip

\noindent{\bf Acknowledgments} This work is partly supported by Natural Scientific Foundation of China (Grant No. 11631006, No. 11790272 and No.11471238). The authors thank Qinbo Chen and Hitoshi Ishii for for helpful discussion.

\section{Representation formula and vanishing contact structure}

We will study \eqref{eq:intro_HJe} when $M=\R^n$. It is not difficult to see that the associated Lagrangian $L$ defined in \eqref{eq:intro_L} is a function of $C^2$ class and it satisfies the following conditions:
\begin{enumerate}[(L1)]
  \item $L_{vv}(x,r,v)>0$ for all $(x,r,v)\in \R^n\times\R\times\R^n$.
  \item For each $r\in\R$, there exist two superlinear and nondecreasing function $\overline{\theta}_r,\theta_r:[0,+\infty)\to[0,+\infty)$, $\theta_r(0)=0$ and $c_r>0$, such that 
  $$
  \overline{\theta}_r(|v|)\geqslant L(x,r,v)\geqslant\theta_r(|v|)-c_r,\quad (x,v)\in\R^n\times\R^n.
  $$
  \item There exists $K>0$ such that
  $$
  |L_r(x,r,v)|\leqslant K,\quad (x,r,v)\in \R^n\times\R\times\R^n.
  $$
\end{enumerate}
Let $\{L^{\lambda}\}_{\lambda>0}$ be a family of Tonelli Lagrangians satisfying conditions (L1)-(L3). We denote by $H^{\lambda}$ the associated Hamiltonians. Set $L_0(x,v):=L(x,0,v)$ and $L_{\lambda}(x,v):=L^{\lambda}(x,0,v)$. For the family $\{L^{\lambda}\}_{\lambda>0}$ we also need the following conditions:
\begin{enumerate}[(L3')]
	\item $|L^{\lambda}_u|\leqslant K_{\lambda}$ with $\lim_{\lambda\to0^+}K_{\lambda}=0$ and $\lim_{\lambda\to0^+}L_{\lambda}(x,v)=L_0(x,v)$ for all $(x,v)$. Moreover, for any compact subset $S\subset\R^n\times\R^n$, there exists a constant $C_S$ such that $\sup_{(x,v)\in S}|L_{\lambda}(x,v)|\leqslant C_S$.
\end{enumerate}
\begin{enumerate}[(L3'')]
	\item $|L^{\lambda}_u|\leqslant K_{\lambda}$ with $\lim_{\lambda\to0^+}K_{\lambda}=0$ and $L_{\lambda}$ tends to $L_0$ uniformly as $\lambda\to0^+$ on any compact subset of $\R^n\times\R^n$.
\end{enumerate}

\subsection{Representation formulae for fundamental solutions and viscosity solutions} 

In this section, we want to give a new representation formula for the viscosity of the Hamilton-Jacobi equation \eqref{eq:intro_HJe} with $H$ satisfying condition (H1)-(H3). Such a representation formula for Tonelli systems appeared first in \cite{Wang-Wang-Yan1} and \cite{Wang-Wang-Yan2} using an implicit variational principle and a fixed point method. In \cite{CCY}, the authors give an alternative approach based on Hergoltz' variational principle. Our following new representation formula for the fundamental solutions is motivated by the multiplier rule (see \cite{Clarkebook3}). 

\begin{The}
If $L$ satisfies conditions \mbox{\rm (L1)-(L3)}, then we have that
\begin{equation}\label{eq:fundamental_solution2}
	A(t,x,y,u)=\inf_{\xi\in\Gamma^t_{x,y}}\left\{e^{\int^t_0L_u d\tau}u+\int^t_0e^{\int^t_sL_u d\tau}(L-u_{\xi}\cdot L_u)\ ds\right\},
\end{equation}
where $u_{\xi}$ is uniquely determined by \eqref{eq:caratheodory_L}.
\end{The}

\begin{Rem}
If $L(x,u,v)=-\lambda u+L_0(x,v)$ for $\lambda>0$, then  $L_u=-\lambda$ and $L-u\cdot L_u\equiv L_0$. Therefore, the curve $u_{\xi}$ given by \eqref{eq:caratheodory_L} does not appear in the representation formula above except for the initial point $u_{\xi}(0)$.
\end{Rem}

\begin{proof}
	Fix $x,y\in\R^n$, $t>0$ and $u\in\R$. For any $\xi\in\Gamma^t_{x,y}$ with $u_{\xi}$ determined by \eqref{eq:caratheodory_L}, $u_{\xi}(0)=u$, we have that
	\begin{align*}
		u_{\xi}(t)=&\,u+\int^t_0L(\xi,u_{\xi},\dot{\xi})\ ds\\
		=&\,u+\int_0^t e^{-\int_0^s L_u(\xi,u_{\xi},\dot{\xi})d\tau} L(\xi,u_{\xi},\dot{\xi})+\left(1-e^{-\int_0^s L_u(\xi,u_{\xi},\dot{\xi})d\tau}\right)\dot{u}_{\xi} ds\\
		=&\,u+\int^t_0e^{-\int^s_0 L_u\ d\tau}(L-u_{\xi}\cdot L_u)\ ds+(1-e^{-\int^t_0L_u d\tau})u_{\xi}(t).
	\end{align*}
	Therefore
	\begin{align*}
		e^{-\int^t_0L_u d\tau}u_{\xi}(t)=u+\int^t_0e^{-\int^s_0 L_u\ d\tau}(L-u_{\xi}\cdot L_u)\ ds,
	\end{align*}
	or, equivalently,
	\begin{align*}
		u_{\xi}(t)=e^{\int^t_0L_u d\tau}u+\int^t_0e^{\int^t_sL_u d\tau}(L-u_{\xi}\cdot L_u)\ ds.
	\end{align*}
	Since $A(t,x,y,u)=\inf_{\xi\in\Gamma^t_{x,y}}u_{\xi}(t)$, then we obtain \eqref{eq:fundamental_solution2}.
\end{proof}


\begin{The}\label{representation_formula_2}
If $L$ satisfies conditions \mbox{\rm (L1)-(L3)}, or equivalently, $H$ satisfies conditions \mbox{\rm (H1)-(H3)}, and $\phi$ is bounded and Lipschitz real-valued function on $\R^n$ with a Lipschitz constant $\mbox{\rm Lip}\,(\phi)$, then 
\begin{equation}\label{eq:rep2}
	u(t,x)=\inf_{y\in\R^n}A(t,y,x,\phi(y)),\quad t>0, x\in\R^n,
\end{equation}
is solution of \eqref{eq:intro_HJe} in the viscosity sense where $A(t,x,y,u)$ is given by \eqref{eq:fundamental_solution2}.
\end{The}

We will postpone the proof of Theorem \ref{representation_formula_2}. To verify that the infimum in \eqref{eq:rep2} is indeed a minimum, we want to show the boundedness of the set
\begin{equation}\label{eq:sublevel}
	\Lambda_t^x:=\{y\in\R^n: A(t,y,x,\phi(y))-A(t,x,x,\phi(x))\leqslant 0\}.
\end{equation}
For this purpose we need a refinement of Lemma \ref{bound_u}. Notice that when one works on a closed manifold instead of $\R^n$, at least the the infimum can be achieved automatically. But a quantitative estimate on the size of ball containing $\Lambda_t^x$ has its own interest.

\begin{Lem}\label{bound_fundamental_solution}
	Suppose that $L$ satisfies conditions \mbox{\rm (L1)-(L3)}. Given $x,y\in\R^n$, $t>0$ and $u\in\R$. Then we have that
	\begin{align*}
		A(t,x,y,u)\geqslant e^{Kt}u+e^{-Kt}\int^t_0\theta_0(|\dot{\xi}(s|)\ ds-c_0te^{Kt},\quad\text{if}\quad u\leqslant 0,
	\end{align*}
	and 
	\begin{align*}
		A(t,x,y,u)\geqslant e^{-Kt}u+e^{-Kt}\int^t_0\theta_0(|\dot{\xi}(s|)\ ds-c_0te^{Kt},\quad\text{if}\quad u\geqslant 0.
	\end{align*}
	Moreover, for $C=\overline{\theta}_0(0)$, we obtain that
	\begin{align*}
		A(t,x,x,u)\leqslant e^{-Kt}u+Cte^{Kt},\quad\text{if}\quad u\leqslant 0,
	\end{align*}
	and
	\begin{align*}
		A(t,x,x,u)\leqslant e^{Kt}u+Cte^{Kt},\quad\text{if}\quad u\geqslant 0.
	\end{align*}
\end{Lem}

\begin{Rem}
Notice that if $L$ satisfies conditions (L1)-(L3), then we have that, for any $(x,u,v)\in\R^n\times\R\times\R^n$, we have that
\begin{equation}\label{eq:DeltaL}
	\begin{split}
		-Ku\leqslant&\,L(x,u,v)-L(x,0,v)\leqslant Ku,\quad u\geqslant0,\\
	   Ku\leqslant&\,L(x,u,v)-L(x,0,v)\leqslant -K u,\quad u\leqslant0.
	\end{split}
\end{equation}	
\end{Rem}

\begin{proof}
	Given $x,y\in\R^n$, $t>0$ and $u\in\R$. Let $\xi\in\Gamma^t_{x,y}$ be a minimizer for $A(t,x,y,u)$ where $u_{\xi}$ is determined by \eqref{eq:caratheodory_L} with initial condition $u_{\xi}(0)=u$. Set $E_0=\{s\in(0,t): u_{\xi}(s)=0\}$, $E_+=\{s\in(0,t): u_{\xi}(s)>0\}$ and $E_-=\{s\in(0,t): u_{\xi}(s)<0\}$, then we have that $E_+$ (resp. $E_-$) is a union of a finite or countable collection of open intervals $\{(a_i,b_i)\}$ (resp. $\{(c_j,d_j)\}$) which mutually intersect empty. Notice that if $E_+$ (resp. $E_-$) is nonempty then $u_{\xi}(a_i)=u_{\xi}(b_i)=0$ (resp. $u_{\xi}(c_j)=u_{\xi}(d_j)=0$) unless $a_i=0$ or $b_i=t$ (resp. $c_j=0$ or $d_j=t$). Since $\dot{u}_{\xi}(s)=0$ for almost all $s\in E_0$\footnote{See, for instance, Theorem 19 in \cite{Varberg}.}, then we have that
	\begin{equation}\label{eq:E_0}
		0=\int_{E_0}\dot{u}_{\xi}\ ds=\int_{E_0}L_0(\xi,\dot{\xi})\geqslant\int_{E_0}\theta_0(|\dot{\xi}(s)|)\ ds-c_0|E_0|.
	\end{equation}
	Due to \eqref{eq:DeltaL}, $\dot{u}_{\xi}(s)\geqslant L_0(\xi(s),\dot{\xi}(s))-Ku_{\xi}(s)$ for all $s\in E_+$, thus
	\begin{equation}\label{eq:E_+}
		\begin{split}
			e^{K\bar{b}}u_{\xi}(\bar{b})-e^{K\bar{a}}u_{\xi}(\bar{a})=&\,\int_{E_+}\frac d{ds}\{e^{Ks}u_{\xi}(s)\}\ ds\geqslant\int_{E_+}e^{Ks}L_0(\xi(s),\dot{\xi}(s))\ ds\\
		\geqslant&\,\int_{E_+}\theta_0(|\dot{\xi}(s)|)\ ds-c_0|E_+|e^{K\bar{b}}
		\end{split}
	\end{equation}
	where $\bar{a}=\inf_ia_i$ and $\bar{b}=\sup_ib_i$. Similarly, we also have that $\dot{u}_{\xi}(s)\geqslant L_0(\xi(s),\dot{\xi}(s))+Ku_{\xi}(s)$ for all $s\in E_-$
	\begin{equation}\label{eq:E_-}
		\begin{split}
			e^{-K\bar{d}}u_{\xi}(\bar{d})-e^{-K\bar{c}}u_{\xi}(\bar{c})=&\,\int_{E_-}\frac d{ds}\{e^{-Ks}u_{\xi}(s)\}\ ds\geqslant\int_{E_-}e^{-Ks}L_0(\xi(s),\dot{\xi}(s))\ ds\\
		\geqslant&\,e^{-Kt}\int_{E_-}\theta_0(|\dot{\xi}(s)|)\ ds-c_0|E_-|
		\end{split}
	\end{equation}
	where $\bar{c}=\inf_jc_j$ and $\bar{d}=\sup_jd_j$.
	
	Now, invoking \eqref{eq:E_0}, \eqref{eq:E_+} and \eqref{eq:E_-}, we have the following relations (which is not optimal but it is enough for our purpose): if $u\leqslant 0$,  
	\begin{align*}
		u_{\xi}(t)\geqslant e^{Kt}u+e^{-Kt}\int^t_0\theta_0(|\dot{\xi}(s|)\ ds-c_0te^{Kt},
	\end{align*}
	and, if $u\geqslant 0$, we have that
	\begin{align*}
		u_{\xi}(t)\geqslant e^{-Kt}u+e^{-Kt}\int^t_0\theta_0(|\dot{\xi}(s|)\ ds-c_0te^{Kt}.
	\end{align*}
	This completes the proof of the first part of the lemma.
	
	Now, let $y=x$ and $\xi(s)\equiv x$ for all $s\in[0,t]$. Recalling $|L(x,0)|\leqslant C$, we have that
	\begin{equation}\label{eq:E_02}
		0=\int_{E_0}\dot{u}_{\xi}\ ds=\int_{E_0}L_0(\xi,\dot{\xi})\ ds=\int_{E_0}L_0(x,0)\ ds\leqslant C|E_0|.
	\end{equation}
	By \eqref{eq:DeltaL}, we obtain that
	\begin{equation}\label{eq:E_+2}
		\begin{split}
			e^{-K\bar{b}}u_{\xi}(\bar{b})-e^{-K\bar{a}}u_{\xi}(\bar{a})=&\,\int_{E_+}\frac d{ds}\{e^{-Ks}u_{\xi}(s)\}\ ds\leqslant\int_{E_+}e^{-Ks}L_0(x,0)\ ds\\
		\leqslant&\,C|E_+|,
		\end{split}
	\end{equation}
	and 
	\begin{equation}\label{eq:E_-2}
		\begin{split}
			e^{K\bar{d}}u_{\xi}(\bar{d})-e^{K\bar{c}}u_{\xi}(\bar{c})=&\,\int_{E_-}\frac d{ds}\{e^{Ks}u_{\xi}(s)\}\ ds\leqslant\int_{E_-}e^{Ks}L_0(x,0)\ ds\\
		\leqslant&\,C|E_-|e^{K\bar{d}}
		\end{split}
	\end{equation}
	Invoking \eqref{eq:E_02}, \eqref{eq:E_+2}, \eqref{eq:E_-2}, we obtain that
	\begin{align*}
		u_{\xi}(t)\leqslant e^{-Kt}u+Cte^{Kt}
	\end{align*}
	if $u\leqslant0$, and
	\begin{align*}
		u_{\xi}(t)\leqslant e^{Kt}u+Cte^{Kt}
	\end{align*}
	if $u\geqslant0$. This completes the proof the lemma.
\end{proof}

\begin{Lem}\label{inf_min}
	Suppose that $L$ satisfies conditions \mbox{\rm (L1)-(L3)} and $\phi$ is a real-valued bounded and Lipschitz function on $\R^n$. Then the infimum in \eqref{eq:rep2} is attained for every $(t,x)\in(0,+\infty)\times\R^n$. Moreover, there exists $\mu(t)>0$ depending on $\mbox{\rm Lip}\,(\phi)$, such that, for any $(t,x)\in(0,+\infty)\times\R^n$ and any minimal point $y_{t,x}$ of $A(t,\cdot,x,\phi(\cdot))$, we have
	\begin{equation}\label{eq:inf_min}
		|y_{t,x}-x|\leqslant\mu(t)t
	\end{equation}
	where $\mu(t)=Ce^{2Kt}+e^{-2Kt}\theta^*_0(e^{2Kt}(\mbox{\rm Lip}\,(\phi)+1))$ with $C>0$.
\end{Lem}

\begin{proof}
    For any $x,y\in\R^n$ and $t>0$. Let $\xi_y\in\Gamma^t_{y,x}$ be a minimizer of $A(t,y,x,\phi(y))$ and $u_{\xi_y}$ be the unique solution of \eqref{eq:caratheodory_L} with initial condition $u_{\xi_y}(0)=\phi(y)$. Based on the estimates of the lower bound of $A(t,y,x,\phi(y))$ and upper bound of $A(t,x,x,\phi(x))$ in Lemma \ref{bound_fundamental_solution}, we have to deal with estimate of the lower bound of $e^{Kt}\phi(y)-e^{-Kt}\phi(x)$ when $\phi(y),\phi(x)\leqslant0$ which is the most difficult case. Indeed, we have that
    \begin{align*}
    	e^{Kt}\phi(y)-e^{-Kt}\phi(x)=&\,e^{Kt}(\phi(y)-\phi(x))+(e^{Kt}-e^{-Kt})\phi(x)\\
    	\geqslant&\,-e^{Kt}\,\text{Lip}\,(\phi)|y-x|-C_0e^{Kt}(1-e^{-2Kt})
    \end{align*}
    where $C_0=\sup_{x\in\R^n}|\phi(x)|$. Since there exists $C_1>0$ such that $1-e^{-2Kt}\leqslant C_1t$ for all $t\geqslant 0$. Therefore, for $C_2=C_0C_1$, we conclude
    \begin{equation}\label{eq:negative}
    	e^{Kt}\phi(y)-e^{-Kt}\phi(x)\geqslant -e^{Kt}\,\text{Lip}\,(\phi)|y-x|-C_2te^{Kt}.
    \end{equation}
        
    Now, suppose that $\phi(y)\leqslant0$ and $\phi(x)\leqslant0$, then by Lemma \ref{bound_fundamental_solution} and \eqref{eq:negative}, we have that
	\begin{align*}
		&\,A(t,y,x,\phi(y))-A(t,x,x,\phi(x))\\
		\geqslant&\,e^{Kt}\phi(y)+e^{-Kt}\int^t_0\theta_0(|\dot{\xi}(s|)\ ds-c_0te^{Kt}-e^{-Kt}\phi(x)-Cte^{Kt}\\
		=&\,-e^{Kt}\,\text{Lip}\,(\phi)|y-x|+e^{-Kt}\int^t_0\theta_0(|\dot{\xi}(s|)\ ds-(c_0+C_2+C)te^{Kt}.
	\end{align*}
	Therefore, for any $k>0$, we have that
	\begin{align*}
		&\,e^{-Kt}[A(t,y,x,\phi(y))-A(t,x,x,\phi(x))]\\
		\geqslant&\,-\mbox{\rm Lip}\,(\phi)|y-x|+e^{-2Kt}\int^t_0\theta_0(|\dot{\xi}(s)|)\ ds-(c_0+C_2+C)t\\
		\geqslant&\,-\mbox{\rm Lip}\,(\phi)|y-x|+k\int^t_0|\dot{\xi}(s)|\ ds-(c_0+C_2+C+e^{-2Kt}\theta^*_0(ke^{2Kt}))t\\
		\geqslant&\,-\mbox{\rm Lip}\,(\phi)|y-x|+k|y-x|-(c_0+C_2+C+e^{-2Kt}\theta^*_0(ke^{2Kt}))t
	\end{align*}
	Choosing $k=\text{Lip}\,(\phi)+1$ and taking
	\begin{align*}
		\mu(t)=c_0+C_2+C+e^{-2Kt}\theta^*_0(e^{2Kt}(\text{Lip}\,(\phi)+1)),
	\end{align*}
	we have that the set $\Lambda_t^x$ defined in \eqref{eq:sublevel} is contained in $\overline{B(x,\mu(t)t)}$. Therefore $\Lambda_t^x$ is compact and the infimum in \eqref{eq:rep2} is indeed minimum. Moreover, \eqref{eq:inf_min} is a consequence of \eqref{eq:sublevel}. The other cases can be dealt with in a similar way. Indeed, 
	\begin{align*}
		\mu(t)=e^{2Kt}\left(c_0+C_2+C \right)+\theta^*_0((\text{Lip}\,(\phi)+1)),\quad&\,\text{if}\ \phi(y)\geqslant 0, \phi (x)\geqslant 0,\\
		\mu(t)=c_0+C+e^{-2Kt}\theta^*_0(e^{2Kt}(\text{Lip}\,(\phi)+1)),\quad&\,\text{if}\ \phi(y)\leqslant 0, \phi (x)\geqslant 0,\\
		\mu(t)=e^{2Kt}\left(c_0+C \right)+\theta^*_0((\text{Lip}\,(\phi)+1)),\quad&\,\text{if}\ \phi(y)\geqslant 0, \phi (x)\leqslant 0.
	\end{align*}
	This completes the proof.
\end{proof}

\begin{Rem}
It is not clear whether Lemma \ref{inf_min} holds true without the assumption that $\phi$ is bounded in general, while it does for the Lagrangian in the form $L(x,u,v)=-\lambda u+L_0(x,v)$, $\lambda>0$, which is the Lagrangian with respect to the well known discounted Hamiltonian (see, for instance, \cite{DFIZ}). Lemma \ref{inf_min} not only ensures that the infimum in \eqref{eq:rep2} is indeed minimum if $\phi$ is a bounded and Lipschitz continuous function on $\R^n$, but also plays an essential part for the applications to the study of the global propagation of singularities of the associated Hamilton-Jacobi equations (\cite{Cannarsa-Cheng3}, \cite{CCF}, \cite{Chen-Cheng} and \cite{Cannarsa-Cheng4}).
\end{Rem}

\begin{Lem}\label{dyn_prog}
Let $x,y\in\R^n$, $t>0$ and $u\in\R$. For any $\xi\in\Gamma^t_{x,y}$ being a minimizer of \eqref{eq:fundamental_solution}, we denote by $u_{\xi}(s,u)$ the unique solution of \eqref{eq:caratheodory_L} with $u_{\xi}(0,u)=u$. Then, for any $0<t'<t$, the restriction of $\xi$ on $[0,t']$ is a minimizer for 
  $$
  u+\inf\int^{t'}_0L(\xi(s).u_{\xi}(s),\dot{\xi}(s))\ ds
  $$
  with $u_{\xi}$ the unique solution of \eqref{eq:caratheodory_L} restricted on $[0,t']$. Moreover,
  \begin{equation}\label{eq:implcite_2}
  	A(s,x,\xi(s),u)=u_{\xi}(s,u),\quad\forall s\in[0,t],
  \end{equation}
  and $A(s_1+s_2,x,\xi(s_1+s_2),u)=A(s_2,\xi(s_1),\xi(s_1+s_2),u_{\xi}(s_1))$ for any $s_1,s_2>0$ and $s_1+s_2\leqslant t$.	
\end{Lem}

\begin{proof}
	Suppose $x,y\in\R^n$, $t>0$ and $u\in\R$. Let $\xi\in\Gamma^t_{x,y}$ be a minimizer of \eqref{eq:fundamental_solution} and $u_{\xi}(s)=u_{\xi}(s;u)$ be the unique solution of \eqref{eq:caratheodory_L} with $u_{\xi}(0,u)=u$.
	
	Now, let $0<t'<t$. Let $\xi_1\in\Gamma^{t'}_{x,\xi(t')}$ and $\xi_2\in\Gamma^{t-t'}_{\xi(t'),y}$ be the restriction of $\xi$ on $[0,t']$ and $[t',t]$ respectively. Then, we have that
	\begin{align*}
		u_{\xi}(t';u)=&\,u+\int^{t'}_0L(\xi_1(s).u_{\xi_1}(s),\dot{\xi_1}(s))\ ds,\\
		u_{\xi}(t;u)-u_{\xi}(t';u)=&\,\int^t_{t'}L(\xi_2(s).u_{\xi_2}(s),\dot{\xi_2}(s))\ ds.
	\end{align*}
	Then both $\xi_1$ and $\xi_2$ are minimal curve for \eqref{eq:fundamental_solution} restricted on $[0,t']$ and $[t',t]$ respectively by summing up the equalities above and the assumption that $\xi$ is a minimizer of \eqref{eq:fundamental_solution}. In particular, \eqref{eq:implcite_2} follows. The last assertion is direct from the relation 
	$$
	u_{\xi}(s_1+s_2;u)=u_{\xi}(s_2;u_{\xi}(s_1)),\quad \forall s_1,s_2>0,\ s_1+s_2\leqslant t,
	$$
	since $u_{\xi}$ solves \eqref{eq:caratheodory_L}.
\end{proof}

\begin{proof}[Proof of Theorem \ref{representation_formula_2}] 
	We first prove $u$ is a subsolution. Fix $(t_0,x_0)\in(0,+\infty)\times\R^n$. Let $\varphi$ be a $C^1$ test function such that $(t_0,x_0)$ is a local maximal point of $u-\varphi$. That is
	\begin{align*}
		\varphi(t_0,x_0)-\varphi(t,x)\leqslant u(t_0,x_0)-u(t,x),\quad (t,x)\in U,
	\end{align*}
	with $U$ be an open neighborhood of $(t_0,x_0)$ in $\R^n$. Given $v\in\R^n$, then for any $C^1$ curve $\xi:[t_1,t_2]\to U$ with $t_1<t_0<t_2$, $\xi(t_0)=x_0$ and $\dot{\xi}(t_0)=v$, we have that, for any $t \in (t_1,t_2)$, 
	\begin{align*}
	\varphi(t_2,\xi(t_2))-\varphi(t_1,\xi(t_1))\leqslant u(t_2,\xi(t_2))-u(t_1,\xi(t_1)).
	\end{align*}
	Denote $x_i=\xi(t_i)$, $i=1,2$. By Lemma \ref{inf_min}, there exists $y\in\R^n$ and a $C^2$ minimal curve $\eta\in\Gamma^{t_1}_{y,x_1}$ such that
	\begin{align*}
		u(t_1,x_1)=\phi(y)+\int^{t_1}_0L(\eta,u_{\eta},\dot{\eta})\ ds.
	\end{align*}
	For any absolutely continuos curve $\gamma:[t_1,t_2]\to U$ connecting $x_1$ to $x_2$, we define $\eta^*$ on $[0,t_2]$ (which connects $y$ to $x_2$) being the juxtaposition of $\eta$ and $\gamma$. Then,
	\begin{align*}
		u(t_2,x_2)-u(t_1,x_1)\leqslant&\,\left(\phi(y)+\int^{t_2}_0L(\eta^*,u_{\eta^*},\dot{\eta}^*)\ ds\right)-\left(\phi(y)+\int^{t_1}_0L(\eta,u_{\eta},\dot{\eta})\ ds\right)\\
		=&\,\int^{t_2}_{t_1}L(\gamma,u_{\gamma},\dot{\gamma})\ ds.
	\end{align*}
	It follows
	\begin{align*}
		\varphi(t_2,\xi(t_2))-\varphi(t_1,\xi(t_1))\leqslant\int_{t_1}^{t_2}L(\xi,u_{\xi},\dot{\xi})\ ds.
	\end{align*}
	By letting $|t_2-t_1|\to 0$, this gives rise to
	\begin{align*}
		\partial_t \varphi(t_0,x_0)+\partial_x\varphi(t_0,x_0)\cdot v\leqslant L(x_0,u_{\xi}(t_0),v),\quad v\in\R^n.
	\end{align*}
	As an application of Fenchel-Legendre dual and since $u(t_0,x_0)=u_{\xi}(t_0)$, we obtain
	\begin{align*}
		\partial_t \varphi(t_0,x_0)+H(x_0,u(t_0,x_0),\partial_x\varphi(t_0,x_0))\leqslant 0,
	\end{align*}
	which shows that $u$ is a subsolution.
	
	Now we turn to the proof that $u$ is a supersolution. Let $\varphi$ be a $C^1$ test function such that $(t_0,x_0)$ is a local minimal point of $u-\varphi$. That is
	\begin{align*}
		\varphi(t_0,x_0)-\varphi(t,x)\geqslant u(t_0,x_0)-u(t,x),\quad (t,x)\in V,
	\end{align*}
	with $V$ be an open neighborhood of $(t_0,x_0)$ in $\R^n$. Due to Lemma \ref{inf_min} and Lemma \ref{dyn_prog}, there exists a $C^2$ curve $\xi:[t,t_0]\to V$ with $\xi(t_0)=x_0$ such that
	\begin{align*}
		u(t_0,\xi(t_0))-u(t,\xi(t))=u_{\xi}(t_0)-u_{\xi}(t)=\int^{t_0}_{t}L(\xi,u_{\xi},\dot{\xi})\ ds.
	\end{align*}
	Hence
	\begin{align*}
		\varphi(t_0,\xi(t_0))-\varphi(t,\xi(t))\geqslant\int^{t_0}_{t}L(\xi,u_{\xi},\dot{\xi})\ ds.
	\end{align*}
	It follows that
	\begin{align*}
		\partial_t \varphi(t_0,x_0)+\partial_x\varphi(t_0,x_0)\cdot\dot{\xi}(t_0)\geqslant L(x_0,u_{\xi}(t_0),\dot{\xi}(t_0)),
	\end{align*}
	which implies
	\begin{align*}
		\partial_t \varphi(t_0,x_0)+H(x_0,u(t_0,x_0),\partial_x\varphi(t_0,x_0))\geqslant 0,
	\end{align*}
	since $u(t_0,\xi(t_0))=u_{\xi}(t_0)$.
	
	Finally, fix $x\in\R^n$ and let $y_{t,x}$ be any minimizer as in Lemma \ref{inf_min}, we conclude that $\lim_{t\to0^+}y_{t,x}=x$. Thus,
	\begin{align*}
		u(t,x)-\phi(x)=\phi(y_{t,x})-\phi(x)+\int^t_0L(\xi,u_{\xi},\dot{\xi})\ ds
	\end{align*}
	Therefore,
	\begin{align*}
		|u(t,x)-\phi(x)|\leqslant\text{Lip}\,(\phi)|y_{t,x}-x|+t\max_{s\in[0,t]}|L(\xi(s),u_{\xi}(s),\dot{\xi}(s))|.
	\end{align*}
	Since Proposition \ref{bound_u} and Proposition \ref{Lip} and $\phi$ is bounded, then, for $0<t\leqslant 1$, we conclude that, there exists $C_1>0$ independent of $x$, $t$ and $u$, such that
	\begin{align*}
		\max_{s\in[0,t]}\{|\xi(s)|,|u_{\xi}(s)|,|\dot{\xi}(s)|\}\leqslant C_1.
	\end{align*}
	It follows that there exists $C_2>0$ such that $\max_{s\in[0,t]}|L(\xi(s),u_{\xi}(s),\dot{\xi}(s))|\leqslant C_2$, and
	\begin{align*}
		|u(t,x)-\phi(x)|\leqslant\text{Lip}\,(\phi)|y_{t,x}-x|+tC_2.
	\end{align*}
	This leads our conclusion that $\lim_{t\to0^+}u(t,x)=\phi(x)$ and completes the proof. 
\end{proof}

\subsection{Vanishing contact structure}

Let $u^{\lambda}$ be the viscosity solution of \eqref{eq:intro_HJe} with respect to $H^{\lambda}$ defined by Herglotz' variational principle \eqref{eq:fundamental_solution} under the constrain \eqref{eq:caratheodory_L}. If $H_0$ is the Fenchel-Legendre dual of $L_0$, then $u$ defined by
\begin{equation}\label{eq:value_function_no_u}
	u(t,x)=\inf_{y\in\R^n}\{\phi(y)+A_t(y,x)\},\quad x\in\R^n, t>0,
\end{equation}
is a viscosity solution of 
\begin{equation}\label{eq:HJe1}\tag{HJ'$_e$}
\begin{cases}
	D_tu(t,x)+H_0(x,D_xu(t,x))=0,& (t,x)\in(0,+\infty)\times\R^n\\
	u(0,x)=\phi(x),& x\in\R^n,
\end{cases}
\end{equation}
where $A_t(x,y)$ is the fundamental solution or the least action with respect to $L_0$.

In this section, we will begin with an easier problem to show, for Cauchy problem, the vanishing discount problem mentioned in the introduction can be generalized to that of vanishing contact structure. 

\begin{Lem}\label{bound_u_2}
Suppose $L$ satisfies conditions \mbox{\rm (L1)-(L3)}. Given $x\in\R^n$, $t,R>0$, $u\in\R$ and $|y-x|\leqslant R$. If $\xi\in\Gamma^t_{y,x}$ is a minimizer for $A_t(y,x)$ and $u_{\xi}$ is determined by \eqref{eq:caratheodory_L} with respect to $L^\lambda$ and $\xi$, then we have that
\begin{align*}
	|u^\lambda_{\xi}(s)|\leqslant tF(t,R/t)+C(t)|u|+\exp(K_\lambda t)\cdot\int_0^t|L_\lambda(\xi,\dot{\xi})-L_0(\xi,\dot{\xi})|\ ds,\quad s\in[0,t].
\end{align*}
where $F(t,R/t)=(\overline{\theta}_0(R/t)+2c_0)\exp(K_\lambda t)$, $C(t)=tK_\lambda \exp(K_\lambda t)+1$
\end{Lem}

\begin{proof}
Denoting by $\xi_0 \in \Gamma^t_{y,x}$ the straight line segment defined by $\xi_0(s)=y+s(y-x)/t$ for any $s\in[0,t]$ and in view of (L2), we have that
	\begin{align*}
	 &\, \int_0^t| L_0(\xi,\dot{\xi})| \ ds\leqslant \int_0^t\{L_0(\xi,\dot{\xi})+2c_0\}\ ds \leqslant  \int_0^t\{L_0(\xi_0,\dot{\xi_0})+2c_0\}\ ds \\
	\leqslant&\,\int_0^t\{\overline{\theta}_0(R/t)+2c_0\}\ ds=t\kappa(R/t),
	\end{align*}
	where $\kappa(r)=\overline{\theta}_0(r)+2c_0$. Therefore, 
	\begin{align*}
	&\,|u^{\lambda}_{\xi}(s)-u|\leqslant\int_0^s|L^{\lambda}(\xi,u^{\lambda}_{\xi},\dot{\xi})|\ d\tau\leqslant\int_0^s|L_{\lambda}(\xi,\dot{\xi})|\ d\tau+ K_{\lambda}\int_0^s|u_\xi^\lambda|\ d\tau \\
	\leqslant&\,\int_0^s|L_0(\xi,\dot{\xi})|\ d\tau+\int_0^s|L_\lambda(\xi,\dot{\xi})-L_0(\xi,\dot{\xi})|\ d\tau+ K_{\lambda}\int_0^s |u_\xi^\lambda-u| \ d\tau +t K_\lambda |u| \\
	\leqslant&\,\int_0^t|L_0(\xi,\dot{\xi})|\ ds+\int_0^t|L_\lambda(\xi,\dot{\xi})-L_0(\xi,\dot{\xi})|\ ds+ K_{\lambda}\int_0^s |u_\xi^\lambda-u|\ d\tau +t K_\lambda |u| \\
	\leqslant&\, t\kappa(R/t)+\int_0^t|L_\lambda(\xi,\dot{\xi})-L_0(\xi,\dot{\xi})|\ ds+K_{\lambda}\int_0^s |u_\xi^\lambda-u| \ d\tau +t K_\lambda |u|  
	\end{align*}
	Due to Gronwall inequality , we obtain that
	\begin{align*}
	|u^{\lambda}_{\xi}(s)-u|\leqslant\left(t\kappa(R/t)+tK_\lambda|u|+\int_0^t|L_\lambda(\xi,\dot{\xi})-L_0(\xi,\dot{\xi})|\ ds\right) \exp(K_\lambda t),
	\end{align*}
	which leads to our conclusion.
\end{proof}

\begin{The}\label{vanishing_1}
Suppose that $\{L^{\lambda}\}_{\lambda>0}$ is a family of Tonelli Lagrangians satisfying conditions \mbox{\rm (L1)}, \mbox{\rm (L2)} and \mbox{\rm (L3')}, with $\{H^{\lambda}\}$ the family of associated Tonelli Hamiltonians. Let each $u^{\lambda}$ be the unique viscosity solution of \eqref{eq:intro_HJe} with respect to $H^{\lambda}$ and $u$ defined by \eqref{eq:value_function_no_u} be the unique viscosity solution of \eqref{eq:HJe1}. If $\phi$ is Lipschitz and bounded, then
\begin{align*}
	\lim_{\lambda\to0^+}u^{\lambda}(t,x)=u(t,x),\quad (t,x)\in(0,+\infty)\times\R^n.
\end{align*}
\end{The}

\begin{Rem}
	For the uniqueness of the viscosity solutions of both \eqref{eq:intro_HJe} and \eqref{eq:HJe1}, see, for instance, \cite{Barles}.
\end{Rem}

\begin{proof}
	Fix $x\in\R^n$ and $t>0$. Let $y=y_{t,x}$ be a minimizer for $A_t(\cdot,x)$ with respect to $L_0$. Then there exists a $C^2$ curve $\xi:[0,t]\to\R^n$, $\xi(0)=y$ and $\xi(t)=x$, such that
	\begin{align*}
		u(t,x)=\phi(y)+\int^t_0L_0(\xi,\dot{\xi})\ ds.
	\end{align*}
	Let $u^{\lambda}_{\xi}$ be the unique solution of \eqref{eq:caratheodory_L} with respect to $\xi$ and $L^{\lambda}$, then we have that
	\begin{align*}
		u^{\lambda}(t,x)-u(t,x)\leqslant&\,\left\{\phi(y)+\int^t_0L^{\lambda}(\xi,u^{\lambda}_{\xi},\dot{\xi})\ ds\right\}-\left\{\phi(y)+\int^t_0L_0(\xi,\dot{\xi})\ ds\right\}\\
		=&\,\int^t_0\{L^{\lambda}(\xi,u^{\lambda}_{\xi},\dot{\xi})-L_0(\xi,\dot{\xi})\}\ ds\\
		=&\,\int^t_0\{L^{\lambda}(\xi,u^{\lambda}_{\xi},\dot{\xi})-L_{\lambda}(\xi,\dot{\xi})\}\ ds+\int^t_0\{L_{\lambda}(\xi,\dot{\xi})-L_0(\xi,\dot{\xi})\}\ ds\\
		 \leqslant&\,tK_{\lambda}\max_{s\in[0,t]}|u^{\lambda}_{\xi}(s)|+\int^t_0\{L_{\lambda}(\xi,\dot{\xi})-L_0(\xi,\dot{\xi})\}\ ds
	\end{align*}
	In the same way as the proof of Lemma \ref{inf_min} (see also \cite{Cannarsa-Cheng3}), we can get $|y-x|=|y_{t,x}-x|\leqslant \mu_0 t$ ,where $\mu_0=c_0+C+\theta_0^*(\text{Lip}\,(\phi)+1)$.
	
	Now, invoking Lemma \ref{bound_u_2}, for $\lambda>0$ such that $K_{\lambda}\leqslant1$, we conclude
	\begin{align*}
	\max_{s\in[0,t]}|u^{\lambda}_{\xi}(s)|\leqslant tF_1(t,\mu_0)+C(t)\max_{y\in \R^n}|\phi(y)|+\exp(K_\lambda t)\cdot\int_0^t|L_\lambda(\xi,\dot{\xi})-L_0(\xi,\dot{\xi})|\ ds.
	\end{align*}
	Since both $\xi$ and $\dot{\xi}$ is uniformly bounded (see, for instance, \cite{Cannarsa-Cheng3}), by using (L3') and dominated convergence theorem, we have that
	\begin{equation}\label{eq:limsup}
		\limsup_{\lambda\to0^+}(u^{\lambda}(t,x)-u(t,x))\leqslant0.
	\end{equation}

	Similarly, let $y^{\lambda}=y^{\lambda}_{t,x}$ be a minimizer for $A(t,\cdot,x,\phi(\cdot))$. Then there exists a $C^2$ curve $\eta^{\lambda}:[0,t]\to\R^n$, $\eta^{\lambda}(0)=y^{\lambda}$ and $\eta^{\lambda}(t)=x$, and a $C^2$ curve $u^{\lambda}_{\eta^{\lambda}}$ determined by \eqref{eq:caratheodory_L} with $u_{\eta^{\lambda}}(0)=\phi(y^{\lambda})$, such that
	\begin{align*}
		u^{\lambda}(t,x)=\phi(y^{\lambda})+\int^t_0L^{\lambda}(\eta^{\lambda},u^{\lambda}_{\eta^{\lambda}},\dot{\eta}^{\lambda})\ ds.
	\end{align*}
	It follows that
	\begin{align*}
		&\,u(t,x)-u^{\lambda}(t,x)\\
		\leqslant&\,\int^t_0\{L_0(\eta^{\lambda},\dot{\eta}^{\lambda})-L^{\lambda}(\eta^{\lambda},u^{\lambda}_{\eta^{\lambda}},\dot{\eta}^{\lambda})\}\ ds\\
		=&\,\int^t_0\{L_0(\eta^{\lambda},\dot{\eta}^{\lambda})-L_{\lambda}(\eta^{\lambda},\dot{\eta}^{\lambda})\}\ ds+\int^t_0\{L_{\lambda}(\eta^{\lambda},\dot{\eta}^{\lambda})-L^{\lambda}(\eta^{\lambda},u^{\lambda}_{\eta^{\lambda}},\dot{\eta}^{\lambda})\}\ ds\\
		\leqslant&\,tK_{\lambda}\max_{s\in[0,t]}|u^{\lambda}_{\eta^{\lambda}}(s)|+\int^t_0\{L_0(\eta^{\lambda},\dot{\eta}^{\lambda})-L_{\lambda}(\eta^{\lambda},\dot{\eta}^{\lambda})\}\ ds.
	\end{align*}
	Invoking Proposition \ref{bound_u}, for $\lambda>0$ such that $K_{\lambda}\leqslant1$, we conclude
	\begin{align*}
		\max_{s\in[0,t]}|u^{\lambda}_{\eta^{\lambda}}(s)|\leqslant tF_2(t,\mu(t))+C(t)\max_{y\in\R^n}|\phi(y)|
	\end{align*}
	where $\mu(t)$ is given by Lemma \ref{inf_min}. Therefore 
	\begin{equation}\label{eq:liminf}
		\limsup_{\lambda\to0^+}(u(t,x)-u^{\lambda}(t,x))\leqslant0.
	\end{equation}
	The combination of \eqref{eq:limsup} and \eqref{eq:liminf} leads to our conclusion.
\end{proof}

\begin{Cor}
Under the same assumptions as Theorem \ref{vanishing_1} and replacing \mbox{\rm (L3')} by \mbox{\rm (L3'')}, then $u^{\lambda}$ tends to $u$ uniformly as $\lambda\to0^+$ on any compact subset of $(0,+\infty)\times\R^n$.	
\end{Cor}

\begin{proof}
	It is similar to the proof of Theorem \ref{vanishing_1}.
\end{proof}

\appendix

\section{Carath\'eodory equations}
Let $\Omega\subset\R^{n+1}$ be an open set. A function $f:\R\times\R^n\to\R^n$ is said to satisfy {\em Carath\'eodory condition} if
\begin{itemize}[-]
  \item for any $x\in\R^n$, $f(\cdot,x)$ is measurable;
  \item for any $t\in\R$, $f(t,\cdot)$ is continuous;
  \item for each compact set $U$ of $\Omega$, there is an integrable function $m_U(t)$ such that
  $$
  |f(t,x)|\leqslant m_U(t),\quad (t,x)\in U.
  $$
\end{itemize}
The classical problem of following Carath\'eodory equation
\begin{equation}\label{eq:caratheodory_system}
	\dot{x}(t)=f(t,x(t)),\quad a.e., t\in I
\end{equation}
is to find an absolutely continuous function $x$ defined on a real interval $I$ such that $(t,x(t))\in\Omega$ for $t\in I$ and satisfies \eqref{eq:caratheodory_system}. 

\begin{Pro}[Carath\'eodory]\label{caratheodory}
	If $\Omega$ is an open set in $\R^{n+1}$ and $f$ satisfies the Carath\'eodory conditions on $\Omega$, then, for any $(t_0, x_0)$ in $\Omega$, there is a solution of \eqref{eq:caratheodory_system} through $(t_0, x_0)$. Moreover, if the function $f(t,x)$ is also locally Lipschitzian in $x$ with a measurable Lipschitz function, then the uniqueness property of the solution remains valid.
\end{Pro}

\noindent For the proof of Proposition \ref{caratheodory} and more results related to Carath\'eodory equation \eqref{eq:caratheodory_system}, the readers can refer to \cite{Coddington_Levinson} and \cite{Filippov}.

\section{Regularity results}

In this section, we collect some fundamental estimates mainly from \cite{CCY}. We always suppose that conditions (L1)-(L3) are satisfied.

\begin{Pro}\label{bound_u}
	Suppose $x\in\R^n$, $t,R>0$, $u\in\R$ and $|y-x|\leqslant R$. If $\xi\in\Gamma^t_{x,y}$ is a minimizer for \eqref{eq:fundamental_solution} with $u_{\xi}$ determined by \eqref{eq:caratheodory_L}, then there exists a function $F:[0,+\infty)\times[0,+\infty)\to[0,+\infty)$ depending on $R$, with $F(r,\cdot)$ being nondecreasing and superlinear and $F(\cdot,r)$ being nondecreasing for any $r\geqslant0$, such that
	\begin{equation}\label{eq:u_bound}
		|u_{\xi}(s)|\leqslant tF(t,R/t)+C(t)|u|,\quad s\in[0,t],
	\end{equation}
	where $C(r)=\max\{e^{Kr},Kre^{Kr}+1\}$. 
\end{Pro}

\begin{Pro}\label{Lip}
	Suppose $x\in\R^n$, $t,R>0$, $u\in\R$ and $|y-x|\leqslant R$. If $\xi\in\Gamma^t_{x,y}$ is a minimizer for \eqref{eq:fundamental_solution}, then, there exist a continuous function $F=F_{u,x}:[0,+\infty)\times[0,+\infty)\to[0,+\infty)$, $F(r,\cdot)$ is nondecreasing and superlinear and $F(\cdot,r)$ is nondecreasing for any $r\geqslant0$, such that
	\begin{align*}
		\operatorname*{ess\ sup}_{s\in[0,t]}|\dot{\xi}(s)|\leqslant F(t,R/t).
	\end{align*}
	In particular,
	\begin{align*}
		|\xi(s)-x|\leqslant tF(t,R/t),\quad s\in[0,t].
	\end{align*}
\end{Pro}

\end{document}